\documentclass{amsproc}
\usepackage{amsmath,amsthm,amsfonts,amssymb,amscd,amsbsy,epsfig,yfonts,bbm, MnSymbol}
\newtheorem{theorem}{Theorem}[section]
\newtheorem{lemma}[theorem]{Lemma}
\newtheorem{problem}[theorem]{Problem}
\newtheorem{proposition}[theorem]{Proposition}

\newtheorem{remark}[theorem]{Remark}

\theoremstyle{definition}

\newcommand{\ZZ}{\mathrm span}

\numberwithin{equation}{section}

\begin{document}
\title[On embeddings of $C_0(K)$ spaces into $C_0(L,X)$ spaces] {On embeddings of $C_0(K)$ spaces into $C_0(L,X)$ spaces}


\author{Leandro Candido}
\address{University of S\~ao Paulo, Department of Mathematics, IME, Rua do Mat\~ao 1010,  S\~ao Paulo, Brazil}
\email{lc@ime.usp.br}
\thanks{The author was supported by FAPESP, process number 2012/15957-6.}

%    General info
%\date{January 1, 1994 and, in revised form, June 22, 1994.}

%\dedicatory{This paper is dedicated for those who will never have a paper dedicated for.}


\begin{abstract} 
Let $C_0(K, X)$ denote the space of all continuous $X$-valued functions defined on the locally compact Hausdorff space $K$ which vanish at infinity, provided with the supremum norm. If $X$ is the scalar field, we denote $C_0(K, X)$ by simply $C_0(K)$. In this paper we prove that for locally compact Hausdorff spaces $K$ and $L$ and for Banach space $X$ containing no copy of $c_0$,  if there is a isomorphic embedding of $C_0(K)$ into $C_0(L,X)$ where either $X$ is separable or $X^*$ has the Radon-Nikod\'ym property, then either $K$ is finite or $|K|\leq |L|$. As a consequence of this result, if there is a isomorphic embedding of $C_0(K)$ into $C_0(L,X)$ where $X$ contains no copy of $c_0$ and $L$ is scattered, then $K$ must be scattered.
\end{abstract}

\maketitle

%%%%%%%%%%%%%%%%%%%%%%%%%%%%%%%%%%%%%%%%%%%%%%%%%%%%%%%%%%%%%%%%%%%%%%%%%%%%%%%%%%%%%%%%%%%%%%%%%%%%%%%%%%%%%%%%%%%%%%%%%%%%
%%%%%%%%%%%%%%%%%%%%%%%%%%%%%%%%%%%%%%%%%%%%%%%%%%%%%%%%%%%%%%%%%%%%%%%%%%%%%%%%%%%%%%%%%%%%%%%%%%%%%%%%%%%%%%%%%%%%%%%%%%%%

\section{Introduction}

For a locally compact Hausdorff space $K$ and a real Banach space $X$, $C_0(K,X)$ denotes the Banach space of all continuous functions $f : K \to X$ which vanish at infinite, provided with the norm: $\|f\| = \sup_{x\in K} \|f (x)\|$. If $X$ is the field of the real numbers $\mathbb{R}$, we denote $C_0(K, X)$ by simply $C_0(K)$. If $K$ is compact these spaces will be denoted by $C(K,X)$ and $C(K)$ respectively. As usual, we also denote $C_0(\mathbb{N})$ by $c_0$. 

For Banach spaces $X$ and $Y$, a linear operator $T : X \to Y$ is called an isomorphic embedding if there are $A,B > 0$ such that
$A \|u\| \leq \|T u \| \leq B \|u\|, \forall \ u\in X.$

If such an embedding exists, we may say that $Y$ contains a copy of $X$ and then write $X\hookrightarrow Y$. On the other hand, we write $X\nhookrightarrow Y$ if $Y$ contains no copy of $X$. An isomorphic embedding of $X$ onto $Y$ is called an isomorphism. Whenever an isomorphism exists, we may say that the spaces are isomorphic and then write $X \sim Y$.

We will also adopt another standard notational conventions. For a Banach space $X$, $B_X$ stands for its unit ball, $S_X$ its unit sphere and $X^*$ its (topological) dual space. We denote by $\ZZ(E)$ the linear span of a set $E\subset X$ and by $\overline{\ZZ}(E)$ the closed linear span of $E$. The cardinality of any set $\varGamma$ will de denoted by $|\varGamma|$.

If $\mathcal{B}_K$ is the Borel $\sigma$-algebra of a locally compact space $K$ and $\mu:\mathcal{B}_K\to X$ is an additive function, the variation of $\mu$ is the positive function $|\mu|:\mathcal{B}_K\to \mathbb{R}$ defined by:
$$|\mu|(B)=\sup\ \sum_{i}\|\mu(B_i)\|, \   B \in \mathcal{B}_K, $$
the supremum being taken over all finite partitions $\{B_i\}$ of $B$ in $\mathcal{B}_K$.

An additive set function $\mu$ is said to have \emph{bounded variation} if $|\mu|(K)<\infty$. The variation $|\mu|$ is clearly additive and if $\mu$ has bounded variation, then $|\mu|$ is $\sigma$-additive if and only if $\mu$ is $\sigma$-additive. We say that $\mu$ is regular if $|\mu|$ is regular in the usual sense. We denote by $M(K,X)$ (or simply $M(K)$, when $X$ is the scalar field) the set of all regular, $\sigma$-additive measures $\mu:\mathcal{B}_K\to X$ of bounded variation. It is not difficult to prove that the map $\mu \mapsto |\mu|(K)$ is a norm making $M(K,X)$, with the usual operations, a Banach space. In this paper, we indentify the space $M(K,X^*)$ with the dual of $C_0(K,X)$ via the \emph{Singer Representation Theorem}, which asserts that there exists an iso\-me\-tric isomorphism between $C_0\left(K,X\right)^*$ and $M(K,X^*)$ such that for every linear functional $\varphi$  and the corresponding measure $\mu$ are related by 
\begin{displaymath}
\langle\varphi, f\rangle =\int f d\mu, \   \  f\in C_0\left(K,X\right),
\end{displaymath}
where the integral is in the sense of Dinculeanu \cite[p. 11]{Dinc1}. This caracterization when $K$ is a compact Hausdorff space can be found in \cite{H}. The locally compact case can be derived from the compact one as explained in \cite[p.2]{C0}.    

For compact Hausdorff spaces $K$ and $L$, a natural question is what properties are transfered from $L$ to $K$ if there is an isomorphism of $C(K)$ onto $C(L)$. A more general question can be posed in the following way: 

\begin{problem}\label{question0}  
Suppose that there is an isomorphic embedding of $C(K)$ into $C(L)$ and $L$ has some property $\mathcal{P}$. Does $K$ has property $\mathcal{P}$?
\end{problem} 

Since the classical paper of Banach \cite{Ba}, there are several fascinating developments related to the questions above (for some very recent, see \cite{Pleb1} and \cite{Pleb}). In the field of vector-valued continuous functions, the following question seems to be a natural extension of the previous one:

\begin{problem}\label{question}  
Suppose that there is an isomorphic embedding of $C(K)$ into $C(L,X)$, where $X$ is a Banach space containing no copy of $c_0$ and $L$ has some property $\mathcal{P}$. Does $K$ has property $\mathcal{P}$?
\end{problem} 

We observe that the Problem \ref{question} makes no sense without the condition: $c_0\nhookrightarrow X$. For if $K$ has some relevant property $\mathcal{P}$ and $L$ doesn't, it is clear that $C(K)$ embeds isomorphically in $C(L,C(K))$.

In the present paper, for locally compact Hausdorff spaces $K$ and $L$, we study isomorphic embeddings of $C_0(K)$ into $C_0(L,X)$ in the spirit of Problem \ref{question}. The results are the following:

\begin{theorem}\label{cardinality}
Let $K$ and $L$ be locally compact Hausdorff spaces and let $X$ be a Banach space containing no copy of $c_0$. If either $X$ is separable or $X^*$ has the Radon-Nikod\'ym property and $C_0(K) \hookrightarrow C_0(L,X)$, then either $K$ is finite or $|K|\leq |L|$. 
\end{theorem}

%%%%%%%%%%%%%%%%%%%%%%%%%%%%%%%%%%%%%%%%%%%%%%%%%%%%%%%%%%%%%%%%%%%%%%%%%%%%%%%%%%%%%%%%%%%%%%%%%%%%%%%%%%%%%%%%%%%%%%%%%%%%

Follows from the previous theorem a extension of the main result of \cite{Candido1}. Another application of Theorem \ref{cardinality} provides us the following result:

%%%%%%%%%%%%%%%%%%%%%%%%%%%%%%%%%%%%%%%%%%%%%%%%%%%%%%%%%%%%%%%%%%%%%%%%%%%%%%%%%%%%%%%%%%%%%%%%%%%%%%%%%%%%%%%%%%%%%%%%%%%%

\begin{theorem}\label{scatt}
Let $K$ and $L$ be locally compact Hausdorff spaces and let $X$ be a Banach space containing no copy of $c_0$. If $C_0(K) \hookrightarrow C_0(L,X)$ and $L$ is scattered, then $K$ is scattered.
\end{theorem}

%%%%%%%%%%%%%%%%%%%%%%%%%%%%%%%%%%%%%%%%%%%%%%%%%%%%%%%%%%%%%%%%%%%%%%%%%%%%%%%%%%%%%%%%%%%%%%%%%%%%%%%%%%%%%%%%%%%%%%%%%%%%

\begin{remark}
Since the unit interval $[0,1]$ is an uncountable perfect set and clearly $$C([0,1])\hookrightarrow C_0(\mathbb{N}, C([0,1])),$$
we conclude that the hypothesis $c_0 \nhookrightarrow X$, in general, cannot be removed neither in Theorem \ref{cardinality} nor in Theorem \ref{scatt}.
\end{remark}

%%%%%%%%%%%%%%%%%%%%%%%%%%%%%%%%%%%%%%%%%%%%%%%%%%%%%%%%%%%%%%%%%%%%%%%%%%%%%%%%%%%%%%%%%%%%%%%%%%%%%%%%%%%%%%%%%%%%%%%%%%%%%
%%%%%%%%%%%%%%%%%%%%%%%%%%%%%%%%%%%%%%%%%%%%%%%%%%%%%%%%%%%%%%%%%%%%%%%%%%%%%%%%%%%%%%%%%%%%%%%%%%%%%%%%%%%%%%%%%%%%%%%%%%%%%

\section{Auxiliary results}
\label{sec}
In order to prove our theorems, we need first to establish some auxiliary results.

\begin{proposition}\label{alpha}
Let $K$ and $L$ be locally compact Hausdorff spaces, $X$ be a Banach space containing no copy of $c_0$ and $T$ be a isomorphic embedding of $C_0(K)$ into $C_0(L,X)$. Given $\epsilon>0$, for every $y \in L$ the set $$K_y(\epsilon)=\bigcup_{\varphi \in S_{X^*}} \left\{x\in K:\left|T^*\left(\varphi \cdot \delta_y\right)\right|\left(\{x\}\right)>\epsilon \right\},$$
where $\delta_y$ denote the Dirac measure centred on $y$, is finite.
\end{proposition}
\begin{proof}
Assume that $K_y(\epsilon)$ is infinite for some $y\in L$. We will show that this assumption leads to a contradiction. Under this hypothesis we may fix distinct points $x_1,x_2,x_3\ldots$ in $K_y$ and $\varphi_1,\varphi_2,\varphi_3\ldots$ in $S_{X^*}$ such that 
$$\left|T^*\left(\varphi_n \cdot \delta_y\right)\right|\left(\left\{x_n\right\}\right)=\left|T^*\left(\varphi_n \cdot \delta_y\right)\left(\{x_n\}\right)\right|>\epsilon, \ \forall \ n\in \mathbb{N}.$$

For every $n$, by regularity of the measure $T^*\left(\varphi_n\cdot \delta_y\right)$, we may fix an open neighborhood $V_n$ of $x_n$ such that 
\begin{equation}\label{rel-77}|T^*\left(\varphi_n\cdot \delta_y\right)|(V_n\setminus \{x_n\})\leq\frac{\epsilon}{2}.\end{equation}

Since $x_1,x_2,x_3\ldots$ are all distinct and $K$ is a locally compact Hausdorff space, by passing to a subsequence if necessary, we may assume a sequence of pairwise disjoint open sets $U_1,U_2,U_3\ldots$ such that $x_n \in U_n\subseteq V_n$ for every $n \in \mathbb{N}$.

By using the Urysohn Lemma, we can take functions $f_n\in C_{0}\left(K\right)$ such that $0\leq f_n\leq 1$, $f_n(x_n)=1$ and $f_n=0$ outside $U_n$.  From (\ref {rel-77}) we have 
\begin{align*}\left\|T f_n\left(y\right) \right\|&\geq \left|\langle\varphi_n,T f_n\left(y\right)\rangle\right|=\left|\int Tf_n d \left(\varphi_n \cdot \delta_y\right)\right|=\left|\int f_n d T^*\left(\varphi_n \cdot \delta_y\right)\right|\\ \notag
&\geq \left| T^*\left(\varphi_n\cdot \delta_y\right)\left(\left\{x_n\right\}\right)\right|-\left|\int f_n d T^*\left(\varphi_n \cdot \delta_y\right) - T^*\left(\varphi_n\cdot \delta_y\right)\left(\left\{x_n\right\}\right) \right|\\ \notag &
>\epsilon -\left|\int f_n d T^*\left(\varphi_n \cdot \delta_y\right) - \int_{\{x_n\}} f_n d T^*\left(\varphi_n \cdot \delta_y\right) \right|\\ \notag &
\geq\epsilon - |T^*\left(\varphi_n\cdot \delta_y\right)|(U_n\setminus \{x_n\})\geq\frac{\epsilon}{2}.\notag
\end{align*}  

Let $S:c_0 \to X$ be the operator defined by $S((a_n)_n)= T(\sum_n a_n\cdot f_n)(y)$. Clearly, $S$ is a bounded linear operator and if $\{e_n:n\in \mathbb{N}\}$ are the unit vectors in $c_0$, we have $\|S(e_{n})\|= \|Tf_{n}(y)\|\geq \epsilon/ {2}$ for every $n \in \mathbb{N}$. We deduce that $\inf_{n \in \mathbb{N}}\|S(e_n)\|\geq \epsilon/ {2}$ and according to a result due to Rosenthal, see \cite[Remark following Theorem 3.4]{Rose1}, there exists an infinite $N\subseteq \mathbb{N}$ such that $S$ restricted to $C_0(N)$ is an isomorphism onto its image. In other words,  $c_0 \hookrightarrow X$, a contradiction.

\end{proof}

%%%%%%%%%%%%%%%%%%%%%%%%%%%%%%%%%%%%%%%%%%%%%%%%%%%%%%%%%%%%%%%%%%%%%%%%%%%%%%%%%%%%%%%%%%%%%%%%%%%%%%%%%%%%%%%%%%%%%%%%%%%%
%%%%%%%%%%%%%%%%%%%%%%%%%%%%%%%%%%%%%%%%%%%%%%%%%%%%%%%%%%%%%%%%%%%%%%%%%%%%%%%%%%%%%%%%%%%%%%%%%%%%%%%%%%%%%%%%%%%%%%%%%%%%
%%%%%%%%%%%%%%%%%%%%%%%%%%%%%%%%%%%%%%%%%%%%%%%%%%%%%%%%%%%%%%%%%%%%%%%%%%%%%%%%%%%%%%%%%%%%%%%%%%%%%%%%%%%%%%%%%%%%%%%%%%%%

The following theorem can be obtained from \cite[Theorem 34, p. 37]{Dinc1} and the definition of the Radon-Nikod\'ym property, see \cite[p. 61]{D}. It will be applied in the proof of the Lemma \ref{beta}. For a locally compact Hausdorff space and a Banach space $X$, a function $\gamma:K\rightarrow X^*$ will be called \emph{weakly$^*$-measurable}, or simply \emph{$w^*$-measurable}, if for every $v\in X$ the numerical function $y\mapsto \langle \gamma(y), v \rangle$ is Borel measurable. It will be called \emph{measurable} if is a pointwise limit of a sequence of Borel measurable functions with finite image. 

%%%%%%%%%%%%%%%%%%%%%%%%%%%%%%%%%%%%%%%%%%%%%%%%%%%%%%%%%%%%%%%%%%%%%%%%%%%%%%%%%%%%%%%%%%%%%%%%%%%%%%%%%%%%%%%%%%%%%%%%%%%%
%%%%%%%%%%%%%%%%%%%%%%%%%%%%%%%%%%%%%%%%%%%%%%%%%%%%%%%%%%%%%%%%%%%%%%%%%%%%%%%%%%%%%%%%%%%%%%%%%%%%%%%%%%%%%%%%%%%%%%%%%%%%

\begin{theorem}\label{simples-continua}
Let $L$ be a localy compact Hausdorff space and $X$ be a Banach space. If either $X$ is separable or $X^*$ has the Radon-Nikod\'ym property, then for every $\mu \in M(L,X^*)$ there is a $w^*$-measurable function $\gamma:L\rightarrow X^*$ satisfying: 
\begin{itemize}
\item[(a)] $\|\gamma(y)\|=1$ for every $y\in L$
\item[(b)] For each $f \in C_0(L,X)$, the function $y\mapsto \langle \gamma(y),f(y) \rangle$ is measurable
\item[(c)] For each $f \in C_0(L,X)$   
$$\int f d\mu = \int \langle \gamma(y),f(y)\rangle d|\mu|(y).$$
\end{itemize}
Moreover, when $X^*$ has the Radon-Nikod\'ym property, $\gamma$ may be chosen to be measurable.  
\end{theorem}

%%%%%%%%%%%%%%%%%%%%%%%%%%%%%%%%%%%%%%%%%%%%%%%%%%%%%%%%%%%%%%%%%%%%%%%%%%%%%%%%%%%%%%%%%%%%%%%%%%%%%%%%%%%%%%%%%%%%%%%%%%%%

In next lemma, we will follow an argument due to G. Plebanek, see \cite[Lemma 3.2]{Pleb1}. 

%%%%%%%%%%%%%%%%%%%%%%%%%%%%%%%%%%%%%%%%%%%%%%%%%%%%%%%%%%%%%%%%%%%%%%%%%%%%%%%%%%%%%%%%%%%%%%%%%%%%%%%%%%%%%%%%%%%%%%%%%%%%

\begin{lemma}\label{beta}
Let $K$ and $L$ be infinite locally compact Hausdorff spaces, $X$ be a Banach space and $T$ be an isomorphic embedding of $C_0(K)$ into $C_0(L,X)$. If either $X$ is separable or $X^*$ has the Radon-Nikod\'ym property, then for each $x \in K$ there is some $y \in L$ and $\varphi\in S_{X^*}$ such that $|T^*\left(\varphi \cdot \delta_y\right)|(\{x\})>0$.
\end{lemma}

\begin{proof} 
Whitout loss of generality, we may assume that $A\|f\|\leq \|T(f)\| \leq \|f\|$, $\  \forall f \in C_0(K)$, for some $A>0$. Pick $x\in K$ and let $\lambda_x\in  T(C_0(K))^*$ be such $\lambda_x(Tf)=f(x)$, $\forall f \in C_0(K)$. It is simple to check that $\|\lambda_x\|\leq 1/A$. We take $\mu_x$, a Hahn-Banach extension of $\lambda_x$ to $M(L,X^*)$.  

According to Proposition \ref{simples-continua}, there exists a $w^*$-measurable function $\gamma:L\to X^*$ (which can be chosen measurable if $X^*$ has the Radon-Nikod\'ym property)  so that
\begin{align}\label{00}
\int f d\mu_x = \int \langle \gamma(y),f(y)\rangle d|\mu_x|(y),\ \forall f \in C_0(L,X).
\end{align} 

Let $0<\epsilon< 1/2$. By applying  \cite[Proposition 3.3]{Candido1}, it is possible to find a compact set $L_0\subseteq L$ satisfying:
\begin{equation*}
\left|\mu_x\right|(L\setminus L_0)\leq \frac{\epsilon}{2}
\end{equation*}
and such that for any $v \in X$, the numerical function $L_0\ni y\mapsto \langle \gamma(y),v \rangle$ is continuous. For each $y\in L$ consider $\varGamma_y=T^*\left(\gamma(y)\cdot \delta_y\right)$. 

\

\textbf{Claim 1} For every $f\in C_0(K,X)$, the function $L_0\ni y\mapsto \int f d \varGamma_y$ is continuous.  

\

Whenever $f\in C_0(K,X)$ and $y,\ y_0\in L_0$ we may write
\begin{align*}
\left|\int f d \varGamma_y - \int f d \varGamma_{y_0} \right|&=\left|\langle \gamma(y), Tf(y) - Tf(y_0)\rangle + \langle \gamma(y) - \gamma(y_0), Tf(y_0)\rangle\right| \\
& \leq \|Tf(y)-Tf(y_0)\|+ \left|\langle \gamma(y) - \gamma(y_0),Tf(y_0)\rangle\right|.\notag
\end{align*}

Since the function $L_0\ni y\mapsto \langle \gamma(y),Tf(y_0) \rangle$ is continuous, by the definiton of $L_0$, it follows from the above relation that $y\mapsto \langle \gamma(y), Tf(y) \rangle$ is continuous. This establishes our claim. 

\

It follows from Claim 1 that
$$L_0\ni y \mapsto \sup_{B_{C_0(K)}} \left|\int f d \varGamma_{y}\right|=|\varGamma_y|(K)$$ is lower semicontinuous. According to the Lusin's theorem, there is a compact set $L_1\subseteq L_0$, such that $\left|\mu_x\right|(L_0\setminus L_1)\leq \epsilon/2$ and the function $L_1\ni y\mapsto |\varGamma_y|(K)$ is continuous.

Next, let $\mathcal{U}_x$ be the collection of all open neighborhoods of $x$, and for each $U \in \mathcal{U}_x$ let $f_{_U} \in C_0(K)$ so that $0\leq f_{_U} \leq 1$, $f_{_U}(x)=1$ and $f_{_U}=0$ outside $U$.

\

\textbf{Claim 2}. For a given $U \in \mathcal{U}_x$, there is $y_{_U} \in L_1$ such that $\langle \gamma(y_{_U}), Tf_{_U}(y_{_U}) \rangle>A/2$.

\

For otherwise, since $|\mu_x|(L_1)\leq |\mu_x|(L)<1/A$ and $|\mu_x|(L\setminus L_1)=|\mu_x|(L\setminus L_0)+|\mu_x|(L_0\setminus L_1)\leq \epsilon/2 +\epsilon/2=\epsilon$, recalling (\ref{00}), we have
\begin{align*}
 1 & =f_{_U}(x) = \int Tf_{_U} d \mu_x = \int \langle \gamma(y),Tf_{_U}(y)\rangle d|\mu_x|(y)\\
   & =\int_{L_1} \langle \gamma(y),Tf_{_U}(y)\rangle d|\mu_x|(y)+\int_{L\setminus L_1} \langle \gamma(y),Tf_{_U}(y)\rangle d|\mu_x|(y)\\
	 &\leq  \frac{A}{2}|\mu_x|(L_1)+ |\mu_x|(L\setminus L_1)\leq \frac{1}{2}+ \epsilon,
\end{align*}
this contradicts the choice of $\epsilon$ and establishes the Claim 2.

\ 

Next, considering $\mathcal{U}_x $ ordered by the reverse inclusion, by the Claim 2, we may fix a net $(y_{_U})_{U}$ in $L_1$ such that $\langle \gamma(y_{_U}), Tf_{_U}(y_{_U}) \rangle > A/2$ $\forall \ U\in \mathcal{U}_x$. Since $L_1$ is compact, passing to a subnet if necessary, we may assume that this net converges to $y_0\in L_1$. 

\

\textbf{Claim 3}. $|\varGamma_{y_0}|(\{x\})>0$. 

\

Let be $U\in \mathcal{U}_x$, where $\overline{U}$ is compact. Since $|\varGamma_{y_0}|$ is regular and $K$ is a locally compact Hausdorff space there are $V_0, V_1 \in \mathcal{U}_x$ with $\overline{V_0}$ and $\overline{V_1}$ compact, satisfying $x\in V_0\subset \overline{V_0}\subset U \subset \overline{U} \subset V_1\subset \overline{V_1}\subset K$ and 
\begin{equation}\label{01}
|\varGamma_{y_0}|(U \setminus \overline{V_0})<\frac{A}{16} \ \text{ and } \ |\varGamma_{y_0}|(K \setminus \overline{V_1})<\frac{A}{16}.
\end{equation}

Consider $f,g\in C_0(K)$ such that $0\leq g \leq 1$, $g=1$ in $\overline{V_0}$ and $g=0$ outside $U$; $0\leq f \leq 1$ and $f=1$ in $\overline{V_1}$.

For any $W\in \mathcal{U}_x$ with  $W \subseteq V_0$ we have $g\geq f_w \geq 0$ and hence
$$\int gd|\varGamma_{y_{_W}}|\geq \int f_w d|\varGamma_{y_{_W}}|\geq \int f_w d\varGamma_{y_{_W}}=\langle \gamma(y_{_W}), Tf_{_W}(y_{_W}) \rangle>\frac{A}{2}.$$

Since the function $L_1\ni y\mapsto |\varGamma_y|(K)$ is continuous, there is $W_0 \subseteq V_0$ such that $|\varGamma_{y_{_W}}|(K)< |\varGamma_{y_{0}}|(K)+A/4$ whenever $W \subseteq W_0$. We infer

$$\int(f-g)d |\varGamma_{y_{_W}}|< |\varGamma_{y_{_W}}|(K)-\frac{A}{2} \leq |\varGamma_{y_{0}}|(K)-\frac{A}{4}, \ \forall \ W \subseteq W_0.$$

By using the the Hahn decomposition theorem for $K$ and $\varGamma_{y_{0}}$ and by using the regularity of $|\varGamma_{y_0}|$, we may find a function $-1\leq h \leq 1$ in $C_0(K)$ such that

$$\int (f-g)d |\varGamma_{y_{0}}|- \frac{A}{32}< \left|\int h (f-g)d \varGamma_{y_{0}} \right|.$$

We know by Claim 1 that $\int h (f-g)d \varGamma_{y_{_W}}$ converges to  $\int h (f-g)d \varGamma_{y_{0}}$, thus, there is $W_1\subseteq W_0$ such that
$$\int (f-g)d |\varGamma_{y_{0}}|- \frac{A}{16}\leq \left|\int h (f-g)d \varGamma_{y_{_W}} \right|\leq \int (f-g)d| \varGamma_{y_{_W}}| \leq |\varGamma_{y_{0}}|(K)-\frac{A}{4},$$
whenever $W\subseteq W_1$. It follows that

$$\frac{A}{16}\leq \left(|\varGamma_{y_{0}}|(K)- \int f d |\varGamma_{y_{0}}|-\frac{A}{16}\right)+ \left(\int g d |\varGamma_{y_{0}}|-\frac{A}{16}\right).$$ 

Recalling (\ref{01}) we deduce that

$$\frac{A}{16}<|\varGamma_{y_{0}}|(U).$$

By the regularity of $|\varGamma_{y_{0}}|$, we conclude that $$0<\frac{A}{16} \leq|\varGamma_{y_{0}}|(\{x\}).$$
\end{proof}

%%%%%%%%%%%%%%%%%%%%%%%%%%%%%%%%%%%%%%%%%%%%%%%%%%%%%%%%%%%%%%%%%%%%%%%%%%%%%%%%%%%%%%%%%%%%%%%%%%%%%%%%%%%%%%%%%%%%%%%%%%%%%%%%

The following lemma will be used in the proof of Theorem \ref{scatt}. 

\begin{lemma}\label{metrizable}
Let $K$ be a scattered compact Hausdorff space and let $X$ and $Y$ be Banach spaces where $Y$ is separable. If $Y\hookrightarrow C(K,X)$, there is a countable metrizable compact space $K_0$ and a separable Banach space $X_0\subset X$ such that $$Y \hookrightarrow C(K_0,X_0)\hookrightarrow C(K,X).$$  
\end{lemma}
\begin{proof}

Since $Y$ is a separable space isomorphically embedded in $C(K,X)$ and $\ZZ\left(\{f \cdot u : f\in C(K), \ 0\leq f\leq 1,\text{ and } u \in X\}\right)$ is dense in $C(K,X)$, it is possible to find sets $\{f_n:n\in \mathbb{N}\} \subset C(K)$, where $0\leq f_n \leq 1$ for each $n\in \mathbb{N}$, and $\{u_n:n\in \mathbb{N}\} \subset X$, such that $Y\hookrightarrow \overline{\ZZ}\left( \{ f_n\cdot u_m: n,m \in \mathbb{N}\}\right).$

Consider the function $\Lambda: K \to [0,1]^{\mathbb{N}}$ defined as $k\mapsto\Lambda(k)=(f_n(k))_{n\in \mathbb{N}}.$ Clearly, $\Lambda$ is continuous, since the projections in each coordinate are continuous. Thus, $K_0=\Lambda(K)\subseteq [0,1]^{\mathbb{N}}$ is compact. Moreover, $K_0$ is metrizable since $[0,1]^{\mathbb{N}}$ is metrizable and is scattered because it is a continuous image of a scattered compact space. Then, $K_0$ must be also countable.  

Let $X_0=\overline{\ZZ}\left(\{ u_n:n \in \mathbb{N}\}\right)$. Clearly, the space $C(K_0,X_0)$ can be isometrically embedded into $C(K,X)$ via the composition map $g\mapsto g\circ \Lambda$. Let $Y_0= \{g\circ \Lambda: g\in C(K_0,X_0)\}$. 

For each $n\in \mathbb{N}$, let $\pi_n:[0,1]^{\mathbb{N}}\to [0,1]$ be the projection on the $n$-th coordinate and let $\lambda_n$ be the restriction of $\pi_n$ to $K_0$. 

For any $m,n \in \mathbb{N}$ fix $g_{m,n}=\lambda_n\cdot u_m \in C(K_0,X_0)$ and observe that
$$g_{m,n}\circ \Lambda = (\lambda_n \circ \Lambda )\cdot u_m = f_n\cdot u_m.$$ It follows that $f_n\cdot u_m\in Y_0$ for every  $m,n \in \mathbb{N}$. Therefore, $$Y\hookrightarrow \overline{\ZZ} \left(\{ f_n\cdot u_m: n,m \in \mathbb{N}\}\right) \subseteq Y_0 \sim C(K_0,X_0) \hookrightarrow C(K,X).$$

\end{proof}

%%%%%%%%%%%%%%%%%%%%%%%%%%%%%%%%%%%%%%%%%%%%%%%%%%%%%%%%%%%%%%%%%%%%%%%%%%%%%%%%%%%%%%%%%%%%%%%%%%%%%%%%%%%%%%%%%%%%%%%%%%%%%%%%
%%%%%%%%%%%%%%%%%%%%%%%%%%%%%%%%%%%%%%%%%%%%%%%%%%%%%%%%%%%%%%%%%%%%%%%%%%%%%%%%%%%%%%%%%%%%%%%%%%%%%%%%%%%%%%%%%%%%%%%%%%%%%%%%
%%%%%%%%%%%%%%%%%%%%%%%%%%%%%%%%%%%%%%%%%%%%%%%%%%%%%%%%%%%%%%%%%%%%%%%%%%%%%%%%%%%%%%%%%%%%%%%%%%%%%%%%%%%%%%%%%%%%%%%%%%%%%%%%

\section{Proofs of the main results}

\begin{proof}[Proof of the Theorem \ref{cardinality}]Let $T$ be an isomorphic embedding of $C_0(K)$ into $C_0(L,X)$. If $L$ is finite, then there is $n\in\mathbb{N}$ such that the space $X\oplus \stackrel{n}{\ldots} \oplus X$ (denoted simply by $X^n$) is isomorphic to $C_0(L,X)$. If $K$ were infinite, then $C_0(K)$, and consequently $X^n$, would have a copy of $c_0$. According to a result due to C. Samuel \cite[Theorem 1]{S}, $X$ would have a copy of $c_0$, a contradiction to our hypothesis. We conclude that if $L$ is finite, then $K$ must be finite as well.

Next, let us assume tha both $L$ and $K$ are infinite. For each $y \in L$ and $n \in \mathbb{N}$ consider $$K_y(1/n)=\bigcup_{\varphi \in S_{X^*}}\left\{x\in K:\left|T^*\left(\varphi \cdot \delta_y\right)\right|\left(\{x\}\right)>1/n\right\},$$
where $\delta_y$ denote the Dirac measure centred on $y$, and $$K_y=\bigcup_{n=1}^{\infty}K_y(1/n).$$

By the Proposition \ref{alpha}, $K_y(1/n)$ is finite for every $n\in \mathbb{N}$ and then $K_y$ is countable for every $y \in L$. According to Lemma \ref{beta}, for every $x \in K$ there exists $y \in L$ such that $x \in K_y$. Thus, $$K\subset \bigcup_{y\in L}K_y$$ and it follows that $$|K|\leq |\bigcup_{y\in L}K_y|\leq |L|\ \omega_0= |L|.$$
\end{proof}

\begin{proof}[Proof of the Theorem \ref{scatt}]
Suppose that $C_0(K) \hookrightarrow C_0(L,X)$ with $L$ scattered and $K$ non-scattered. If $\alpha K=K \stackrel{.}{\cup} \{\infty\}$ is the Aleksandrov one-point compactification of $K$, then $\alpha K$ is also non-scattered and there must be a continuous surjection $\varphi: \alpha K \to [0,1]$, see \cite[Theorem 8.5.4, p. 148]{Se}. 

Assume that $\varphi(\infty)=t\in [0,1]$. The composition map, $f\mapsto f\circ \varphi$, induces an isometric embedding of $C_0([0,1]\setminus\{t\})$ into $C_0(K)$. 

If $\alpha L = L \stackrel{.}{\cup} \{\infty\}$ is the Aleksandrov one-point compactification of $L$, the space $C_0(L,X)$ can be identified isometrically as a subspace of $C(\alpha L,X)$ namely: the subspace of all continuous functions $f:\alpha L\to X$ such that $f(\infty)=0$. We deduce $$C_0([0,1]\setminus\{t\})\hookrightarrow C_0(K)\hookrightarrow C_0(L,X)\hookrightarrow C(\alpha L,X).$$

Since $C_0([0,1]\setminus\{t\})$ is separable and $\alpha L$ is a scattered compact space, according to Lemma \ref{metrizable}, there is a countable metrizable compact space $L_0$ and a separable Banach space $X_0\subset X$ such that $$C_0([0,1]\setminus\{t\}) \hookrightarrow C(L_0,X_0)\hookrightarrow C(\alpha L,X).$$

Since $X_0$ is separable and contains no copy of $c_0$, we may apply the Theorem \ref{cardinality} and obtain $$2^{\omega_0}=|[0,1]\setminus\{t\}|\leq |L_0|=\omega_0,$$
a contradiction, which establishes the theorem.

\end{proof}

%%%%%%%%%%%%%%%%%%%%%%%%%%%%%%%%%%%%%%%%%%%%%%%%%%%%%%%%%%%%%%%%%%%%%%%%%%%%%%%%%%%%%%%%%%%%%%%%%%%%%%%%%%%%%%%%%%%%%%%%%%%%%%%%
%%%%%%%%%%%%%%%%%%%%%%%%%%%%%%%%%%%%%%%%%%%%%%%%%%%%%%%%%%%%%%%%%%%%%%%%%%%%%%%%%%%%%%%%%%%%%%%%%%%%%%%%%%%%%%%%%%%%%%%%%%%%%%%%

\bibliographystyle{amsalpha}

\begin{thebibliography}{A}

\bibitem{Ba}  S. Banach,
\emph{Th\'eorie des op\'erations lin\'eaires}.  
Monografie Matematyczne, Warsaw, 1932.

\bibitem{C0} M. Cambern, 
\textit{Isomorphims of spaces of continuous vector-valued functions}. 
Illinois . Math. 20 (1976), 1--11.

\bibitem{Candido1} L. Candido, E. M. Galego, 
\emph{A weak vector-valued Banach-Stone theorem},
Proc. Amer. Math. Soc. 141 (2013), 3529--3538. 

\bibitem{D} J. Diestel, J.J. Uhl, Jr, 
\emph{Vector Measures}. 
Math. Surveys 15, Amer. Math. Soc., Providence, 1977.
                    
\bibitem {Dinc1} N. Dinculeanu. 
\emph{Vector Integration and Stochastic Integration in Banach Spaces}. 
Wiley Interscience, 2000.

\bibitem{H} W. Hensgen, 
\emph {A simple proof of Singer's representation theorem}. 
Proc. Amer. Math. Soc. 124 (1996), 3211--3212.

\bibitem{Pleb1} G. Plebanek,
\emph{On isomorphisms of Banach spaces of continuous functions}.
arXiv:0706.1234v1 [math.FA].

\bibitem{Pleb} G. Plebanek,
\emph{On positive embeddings of $C(K)$ spaces}.
Studia Math. 216 (2013), 179--192.

\bibitem{Rose1}  H. P. Rosenthal,
\emph{On relatively disjoint families of measures, with some applications to Banach space theory}.  
Studia Math. 37 (1970), 13--30.

\bibitem{S} C. Samuel, 
\emph {Sur la reproductibilite des espaces $l_p$}.  
Math. Scand. 45 (1979), 103--117.

\bibitem{Se}  Z. Semadeni,
\emph{Banach Spaces of Continuous Functions Vol. I},  
Monografie Matematyczne, Tom 55. Warsaw, PWN-Polish Scientinfic Publishers, Warsaw, 1971.

\end{thebibliography}

\end{document}